\documentclass[12pt,english]{amsart}
\usepackage[T1]{fontenc}
\usepackage[latin9]{inputenc}
\usepackage{geometry}
\usepackage{enumitem}
\usepackage{amstext}
\usepackage{amsthm}
\usepackage{amssymb}
\usepackage{setspace}
\usepackage{xcolor}
\makeatletter
\numberwithin{figure}{section}
\theoremstyle{plain}
\newtheorem{thm}{\protect\theoremname}[section]
\theoremstyle{remark}
\newtheorem*{acknowledgement*}{\protect\acknowledgementname}
\theoremstyle{plain}

\theoremstyle{plain}
\newtheorem{lem}[thm]{\protect\lemmaname}
\theoremstyle{plain}
\newtheorem{prop}[thm]{\protect\propositionname}
\theoremstyle{remark}

\theoremstyle{plain}

\theoremstyle{remark}

\date{}
\pdfpageattr{/Group << /S /Transparency /I true /CS /DeviceRGB>>}
\usepackage{ae,aecompl}
\usepackage{amsmath}
\usepackage{bbm}
\usepackage{hyperref}
\hypersetup{
	hidelinks = true
}

\let\originalleft\left
\let\originalright\right
\renewcommand{\left}{\mathopen{}\mathclose\bgroup\originalleft}
\renewcommand{\right}{\aftergroup\egroup\originalright}
\usepackage{multicol}
\makeatother

\providecommand{\acknowledgementname}{Acknowledgement}
\providecommand{\factname}{Fact}
\providecommand{\lemmaname}{Lemma}
\providecommand{\propositionname}{Proposition}
\providecommand{\questionname}{Question}
\providecommand{\remarkname}{Remark}
\providecommand{\theoremname}{Theorem}
\providecommand{\defnname}{Definition}

\begin{document}
\global\long\def\RR{\mathbb{R}}%

\global\long\def\CC{\mathbb{C}}%

\global\long\def\ZZ{\mathbb{Z}}%

\global\long\def\NN{\mathbb{N}}%

\global\long\def\QQ{\mathbb{Q}}%

\global\long\def\TT{\mathbb{T}}%

\global\long\def\FF{\mathbb{F}}%

\global\long\def\vphi{\varphi}%

\global\long\def\sub{\subseteq}%

\global\long\def\one{\mathbbm1}%

\global\long\def\vol#1{\text{Vol\ensuremath{\left(#1\right)}}}%

\global\long\def\epi#1{\text{epi}\left\{  #1\right\}  }%

\global\long\def\sp{{\rm sp}}%

\global\long\def\K{\mathcal{K}}%

\global\long\def\A{\mathcal{A}}%

\global\long\def\L{\mathcal{L}}%
\global\long\def\P{\mathcal{P}}%

\global\long\def\W{\mathcal{W}}%

\global\long\def\iprod#1#2{\langle#1,\,#2\rangle}%

\global\long\def\conv{{\rm Conv}}%

\global\long\def\eps{\varepsilon}%

\global\long\def\norm#1{\left\Vert #1\right\Vert }%

\global\long\def\supp#1{{\rm supp}\left(#1\right)}%

\global\long\def\flr#1{\left\lfloor #1\right\rfloor }%

\global\long\def\ceil#1{\left\lceil #1\right\rceil }%

\global\long\def\hseg#1{\left\llbracket #1\right\rrbracket }%

\global\long\def\EE{\mathbb{E}}%
$ $

\global\long\def\RR{\mathbb{R}}%

\global\long\def\ZZ{\mathbb{Z}}%

\global\long\def\ll{\preceq}%

\global\long\def\bp#1{\big(#1\big)}%

\global\long\def\Bp#1{\Big(#1\Big)}%
\global\long\def\supp#1{{\rm supp}\left(#1\right)}%

\global\long\def\card#1{{\rm Card}\left(#1\right)}%

\global\long\def\gge{\succeq}%

\global\long\def\gg{\succ}%

\title{A remark on discrete Brunn-Minkowski type inequalities via transportation of measure}
\author{Boaz A. Slomka}
\address{Department of Mathematics, the Open University of Israel, Ra'anana
	4353701 Israel}
\email{slomka@openu.ac.il}
\begin{abstract}
	We give an alternative proof  for discrete Brunn-Minkowski type inequalities, recently obtained by Halikias, Klartag and the author. This proof also implies somewhat stronger weighted versions of these inequalities. 
	Our approach generalizes  ideas of Gozlan, Roberto, Samson and Tetali from the theory of measure transportation and provides new displacement convexity of entropy type inequalities for the lattice point enumerator.\\ \\
\end{abstract}
\maketitle
\section{Introduction}
In recent years, there has been a growing interest in finding  discrete versions of various results related to  convexity  theory, see e.g., \cite{AHZ17,GG01,Gozlan14,Gozlan2019,Green2020,HLY21,INZ,Klartag2019,MM2020,Matolcsi2020,Melbourne2020, Ollivier-Villani,RYZ17}. %TODO: check if some refs have been published by now.. like MM20 IYN and..
The aim of this note is to give a transport proof for the discrete Brunn-Minkowski type inequalities from \cite{HKS2020}, by extending the results in \cite{Gozlan2019}.In particular, we establish  new entropic versions of these inequalities. 

\subsection{Discrete Brunn-Minkowski inequalities}
We say that an operation $T:\ZZ^n\times\ZZ^n\to\ZZ^n$ {\em admits a Brunn-Minkowski inequality} if for all functions  $f,g,h,k:\ZZ^n\to[0,\infty)$ satisfying  
\begin{equation}\label{DBM_assumption}
f(x)g(y)\le h(T(x,y))k(x+y-T(x,y))\quad\forall x,y\in\ZZ^n, 
\end{equation}
it follows that 
\begin{equation}\label{eq:DBM_conclusion}
\Big(\sum_{x\in\ZZ^n}f(x)\Big) \Big(\sum_{x\in\ZZ^n}g(x)\Big)\le  \Big(\sum_{x\in\ZZ^n}h(x)\Big) \Big(\sum_{x\in\ZZ^n}k(x)\Big).
\end{equation}

One example for such an operation is  $T(x,y)=x\wedge y=\left(\min\left(x_{1},y_{1}\right),\dots,\min\left(x_{n},y_{n}\right)\right)$ which is due to the four functions theorem of Ahlswede and Daykin \cite{Ahlswede1978}. In this case, we have  $x+y-x\wedge y=x\vee y=\left(\max\left(x_{1},y_{1}\right),\dots,\max\left(x_{n},y_{n}\right)\right)$. 
Another example for such an operation is due to the discrete Brunn-Minkowski  inequality of Klartag and Lehec \cite[Theorem 1.4]{Klartag2019}, which  corresponds to $T(x,y)=\flr{ (x+y)/2}$, where $x+y-T(x,y)=\ceil{(x+y)/2}$, 
$\flr x=\left(\flr{x_{1}},\dots\flr{x_{n}}\right)$,  and $\ceil x=\left(\ceil{x_{1}},\dots,\ceil{x_{n}}\right)$. Here $\lfloor r\rfloor = \max \{ m \in \ZZ \, ; \, m \leq r \}$ is the lower integer part of $r\in\RR$ and $\lceil r\rceil=-\lfloor-r\rfloor$ the upper integer part. \smallskip

It was Gozlan, Roberto, Samson and Tetali \cite{Gozlan2019} who have first linked 
between the four functions theorem of Ahlswede and Daykin and the discrete Brunn-Minkowksi inequality of Klartag and Lehec. In their paper, they provided alternative proofs for these results which are based on  ideas from the theory of measure transportation.\smallskip

Recently, a unified elementary proof for the two aforementioned results was given in \cite{HKS2020}. This proof applies to all operations $T:\ZZ^n\times\ZZ^n\to\ZZ^n$ sharing two common properties:
\begin{enumerate}[labelindent=\parindent,leftmargin={*},label=(P\arabic*),align=left]
	\item \label{enu:TEQ }{\em Translation equivariance}: $T\left(x+z,y+z\right)=T\left(x,y\right)+z$ for all $z\in\ZZ^n$.
\item\label{enu:Knothe} {\em Monotonicity in the sense of Knothe}: there exists a decomposition of $\ZZ^n$ into a direct sum of groups $\ZZ^n = G_1 \times \dots \times G_k$ such that  for each $i\in\{1,\dots,k\}$:\smallskip
\begin{enumerate}
	\item[(i)] $T_i:(G_1\times\dots\times G_i)\times(G_1\times\dots\times G_i)\to G_i$ where $T = (T_1,\ldots,T_k)$. In other words, $T_i(x,y)$ depends only on the first $i$ coordinates of its  arguments $x,y \in G_1 \times \dots \times G_k$, so that $T$ is triangular.\smallskip
	\item[(ii)] There exists a total additive ordering $\ll_i$ on $G_i$ such that $T^{(a,b)}_i:G_i\times G_i \to G_i$ defined  by $T_i^{(a,b)}(x,y)=T_i\big((a,x),(b,y)\big)$ for $a,b\in G_1\times\dots\times G_{i-1}$ satisfies
	$$ x_1\ll_i x_2, \ y_1 \ll_i y_2 \qquad \Longrightarrow \qquad  T_i^{(a,b)}(x_1,y_1)\ll_i T_i^{(a,b)}(x_2,y_2) $$
	for all $a,b \in G_1\times\dots\times G_{i-1}$ and $x_1,x_2, y_1, y_2 \in G_i$.
\end{enumerate}
\end{enumerate}
Recall that a total ordering $\preceq$ on an abelian group $G\approx\ZZ^l$ is a binary relation which is reflexive, anti-symmetric and transitive, such that for any distinct $x,y$, either $x\preceq y$ or else $y\preceq x$. An ordering $\preceq$ is additive if for all $x,y,z$, we have $x\ll y\implies x+z\ll y+z$.\smallskip 

Examples for  additive, total orderings on  $\ZZ^{n}$ (or on $\RR^n$) are the standard lexicographic order relation and invertible linear images  thereof.  %Also note that properties \ref{enu:TEQ } and \ref{enu:Knothe} are closed under cartesian products of operations.

\begin{thm}[{\cite[Theorem 1.3]{HKS2020}}]\label{thm:DBM}
Every translation equivariant operation  $T:\ZZ^n\times\ZZ^n\to\ZZ^n$ which is monotone in the sense of Knothe admits a Brunn-Minkowski inequality.
\end{thm}
%We remark that Knothe \cite{knothe} used maps satisfying a condition similar to \ref{enu:Knothe} in his proof of the Brunn-Minkowski inequality.\smallskip

In addition to the four functions theorem and the Brunn-Minkowski inequality of Klartag and Lehec, Theorem \ref{thm:DBM} implies various other inequalities, some of which are related to works of  Ollivier and Villani \cite{Ollivier-Villani}, Iglesias, Yepes Nicol\'as and Zvavitch \cite{INZ}, and Cordero-Erausquin and Maurey \cite{Cordero-Erausquin2017}. For a more detailed account of these implications, see \cite{HKS2020}.\smallskip

Our first main result is the following extension of Theorem \ref{thm:DBM}:
\begin{thm}\label{thm:strongDBM}
	Let $\alpha,\beta,\gamma,\delta>0$ such that $\max\{\alpha,\beta\}\le\min\{\gamma,\delta\}$. Let $T:\ZZ^{n}\times\ZZ^{n}\rightarrow\ZZ^{n}$ satisfy properties \ref{enu:TEQ } and \ref{enu:Knothe} and suppose that
	$f,g,h,k:\ZZ^{n}\to[0,\infty)$ satisfy
	\[
	f^{\alpha}\left(x\right)g^{\beta}\left(y\right)\le h^\gamma\left(T\left(x,y\right)\right)k^\delta\left(x+y-T\left(x,y\right)\right) \qquad \forall x,y\in\ZZ^{n}.
	\]
	Then
	\[
	\Bp{\sum_{x\in\ZZ^{n}}f\left(x\right)}^\alpha\Bp{\sum_{x\in\ZZ^{n}}g\left(x\right)}^{\beta}\le\Bp{\sum_{x\in\ZZ^{n}}h\left(x\right)}^\gamma\Bp{\sum_{x\in\ZZ^{n}}k\left(x\right)}^\delta.
	\]
\end{thm}

 Note that if an operation $T:\ZZ^n\times\ZZ^n\to\ZZ^n$  satisfies properties \ref{enu:TEQ } and \ref{enu:Knothe}, then so does the operation $x+y-T(x,y)$. In the sequel, we shall say that  $T_\pm:\ZZ^n\times\ZZ^n\to\ZZ^n$ are  {\it complementing operations} if they satisfy the above relation, i.e.,  $T_-(x,y)+T_+(x,y)=x+y$.

\subsection{An entropic version}
Our approach is based on the work of Gozlan, Roberto, Samson and Tetali  \cite{Gozlan2019} who proved the next displacement convexity of entropy result for the counting measure $m$ on $\ZZ$ and the operations $T_-(x,y)=\flr{(x+y)/2}$ and  $T_+(x,y)=\ceil{ (x+y)/2}$.

For historical remarks on the   the displacement convexity of entropy property and other discrete variants of it, see e.g., \cite{Gozlan14,Gozlan2019,Ollivier-Villani} and references therein. 

\begin{thm}[{\cite[Theorem 8]{Gozlan2019}}]\label{thm:8} Suppose that $\mu_1,\mu_2$ are finitely supported probability measures on $\ZZ$. Then
	\begin{equation}
	H(\mu_1|m)+H(\mu_2|m)\ge H(\mu_-|m)+H(\mu_+|m)
	\end{equation}
where $H(\mu|m)=\sum_{x\in\ZZ}\mu(x)\log(\mu(x))$ is the relative entropy of $\mu$ with respect to $m$, and $\mu_\pm$ is the push forward of the monotone coupling $\pi$ between $\mu_1$ and $\mu_2$ by $T_\pm$.
\end{thm}

Denote the counting measure on $\ZZ^n$ by $m_n$. The relative entropy of a probability measure $\mu$ on $\ZZ^n$ with respect to $m_n$ is given by $H(\mu|m_n)=\sum_{x\in\ZZ^n}\mu(x)\log(\mu(x))$.\smallskip

Our second main result is the following generalization of Theorem \ref{thm:8}:

\begin{thm}\label{thm:Knothe} Let $T_\pm:\ZZ^n\times\ZZ^n\to\ZZ^n$ be complementing operations satisfying properties \ref{enu:TEQ } and \ref{enu:Knothe}. Suppose that $\mu$ and $\nu$ are finitely supported probability measures on $\ZZ^n$. Then there exists a coupling $\pi$ between $\mu$ and $\nu$ such that, denoting by $\kappa_\pm=\pi\circ{T_\pm}^{-1}$  the push forward of $\pi$ by $T_\pm$, we have
	\begin{equation}\label{eq:9}
H(\mu|m_n)+H(\nu|m_n)	\ge H(\kappa_-|m_n)+H(\kappa_+|m_n).
	\end{equation}
\end{thm}

The coupling $\pi$ for which Eq. \eqref{eq:9} holds is a Knothe coupling which disintegrates into a product of monotone couplings with respect  to the decomposition $\ZZ^n=G_1\times\dots\times G_k$  given in property \ref{enu:Knothe}. The construction of this coupling is described  in  Section \ref{sec:Knothe}.\smallskip

As we shall see,  \cite{Gozlan2019}, Theorem \ref{thm:Knothe} implies Theorem \ref{thm:strongDBM} by  the same duality argument which was used in \cite{Gozlan2019} to prove \cite[Theorem 1.4]{Klartag2019}.
%A similar duality argument was used by Lehec in \cite{Lehec13} to obtain reversed Brascamp-Lieb inequalities. For additional discrete results in the this spirit, see \cite{Gozlan14}, \cite{Ollivier-Villani} and references therein. \smallskip

Theorem \ref{thm:Knothe} is an immediate consequence of  the following extension of \cite[Theorem 9]{Gozlan2019} (which was used to deduce Theorem \ref{thm:8} in the same manner):

\begin{thm}\label{thm:Knothe_9} Let $T_\pm:\ZZ^n\times\ZZ^n\to\ZZ^n$ be complementing operations satisfying properties \ref{enu:TEQ } and \ref{enu:Knothe}.  Suppose that $\mu$ and $\nu$ are finitely supported probability measures on $\ZZ^n$. Then there exists a coupling $\pi$ between $\mu$ and $\nu$ such that, denoting by $\kappa_\pm=\pi\circ{T_\pm}^{-1}$ the push forward of $\pi$ by $T_\pm$, we have
	$$
\sum_{(x,y)\in \ZZ^n\times\ZZ^n}\frac{\kappa_-(T_-(x,y))\kappa_+(T_+(x,y))}{\mu(x)\nu^(y)}\,\pi(x,y)\le 1.
	$$
\end{thm}

The remaining of the paper is organized as follows. In Section \ref{sec:simplified_DCOE} we prove Theorem \ref{thm:Knothe} in the case where $T$ itself is monotone. In Section \ref{sec:Knothe} we use the Knothe coupling to extend the proof of Theorem \ref{thm:Knothe_9} to the general case. Sections \ref{sec:9to8} and \ref{sec:8toBM} are devoted for the derivations of Theorems \ref{thm:Knothe} and \ref{thm:strongDBM}, respectively.

\begin{acknowledgement*}
The author thanks Bo'az Klartag for fruitful conversions and for his advice and comments. The author also thanks Shiri Artstein for her remarks on the written text and the anonymous referee of the paper \cite{HKS2020} for suggesting to pursue this direction. Supported by ISF grant 784/20.
\end{acknowledgement*}

\section{The monotone case} \label{sec:simplified_DCOE}

%In this section we make the preparation towards proving Theorem \ref{thm:Knothe} in the case where $T$ itself is monotone in each of its two entries with respect to some total additive ordering $\ll$ on $\ZZ^n$. This result is given below as Proposition \ref{prop:basic_SCOE}. 

The purpose of this section is to prove Theorem \ref{thm:Knothe_9} in the case where $T_-$ (equivalently $T_+$) itself is monotone in each of its two entries with respect to some total additive ordering $\ll$ on $\ZZ^n$. That is,
\begin{equation}\label{eq:P2'}
x_1\ll x_2, y_1\ll y_2\implies T_\pm(x_1,y_1)\ll T_\pm(x_2,y_2)
\end{equation}
This result is given below as Proposition \ref{prop:9}.

%The proof makes use of the monotone coupling between two measures on $\ZZ^n$ with respect to the given ordering $\ll$ and the properties of  the structure of its support under  the operations $T_-$ and $T_+$. 

The core ideas of the proof  are drawn from the proof of \cite[Theorem \ref{thm:8}]{Gozlan2019}, which is also a particular case of the proposition. However, we also manage to simplify some of its key steps.

% which is mainly thanks to the fact that we consider an abstract operation T rather than a specific one.

%Throughout this section,  Also let $T_\pm:G\times G\to G$ be complementing translation equivariant operations that are non-decreasing in each of their two entries with respect to $\ll$, i.e., if $x_1\ll x_1$ and $y_1\ll y_2$ then $T_\pm(x_1,y_1)\ll T_\pm(x_2,y_2)$. In particular, $T_\pm$ satisfy properties \ref{enu:TEQ } and \ref{enu:Knothe}.

\subsection{The monotone coupling $\pi$ and Proposition \ref{prop:9}} Let  $G\approx\ZZ^n$ be a finitely generated group, endowed with a totally additive ordering $\ll$. Given a probability measure $\mu$ on $G$, the cumulative distribution of $\mu$ with respect to $\ll$  is defined by
$$
F_\mu(x)=\mu((-\infty,x])=\mu\{g\in G\,;\,g\ll x\}\quad\forall x\in G.
$$
Similarly, the generalized inverse of $F_\mu$ at a point $t\in(0,1)$ is given by 
$$F_\mu^{-1}(t)=\inf\{x\in G\,;\,F_\mu(x)\ge t\}.$$

Recall that a coupling between  two probability measures $\mu$ and $\nu$ on $G$ is a probability measure $\pi$ on $G\times G$ whose coordinate marginals are $\mu$ and $\nu$, i.e., for all $A,B\sub G$, we  have $\pi(A,G)=\mu(A)$ and $\pi(G,B)=\nu(B)$.

Given two  finitely supported probability measures   $\mu,\nu$ on $G$ and a random variable  $U$, uniformly distributed on $(0,1)$, we define the monotone coupling between $\mu$ and $\nu$ with respect to the ordering $\ll$ by
$$ \pi={\rm Law}(F_\mu^{-1}(U),F_\nu^{-1}(U) ).$$

One can check that the support of $\pi$ is monotone with respect to $\ll\times\ll$. That is, if $(a,b),(c,d)\in \supp{\pi}$ then either $a\ll c$ and $b\ll d$ or vice versa, $c\ll a$ and $d\ll b$. 
%Indeed, suppose otherwise that  $a\ll c$ and  $d\ll b$ with $(a,b)\neq(c,d)$ and that there exist $t_1,t_2\in(0,1)$ such that $(F_\mu^{-1}(t_1),F_\nu^{-1}(t_1))=(a,b)$ and $(F_\mu^{-1}(t_2),F_\nu^{-1}(t_2))=(c,d)$. Then, on the one hand, $a\ll b$, implies that $t_2>t_1$ and, on the other hand, $d\ll b$ implies that  $t_1>t_2$, a contradiction.

We remark that if $\nu$ stochastically dominates $\mu$, or the other way around, then the coupling $\pi$ is diagonal, i.e., $F_\mu^{-1}(U)\le F_\nu^{-1}(U)$ with probability $1$ or $0$. This is a particular case of Strassen's theorem \cite{Strassen1965}, which holds for partially ordered probability spaces. For more information on this subject, see  e.g., \cite{Lindvall02} and references therein. 

%Using the monotone coupling, ,let us restate Theorem \ref{thm:Knothe} more precisely, in the case where $T_\pm$ are monotone operations:

%\begin{prop}\label{prop:basic_SCOE}Let $G$ be a finitely generated group endowed with a totally additive ordering $\ll$.  Let $T_\pm:G\times G\to G$ be complementing operations satisfying properties \ref{enu:TEQ } and \eqref{eq:P2'}. Suppose that $\mu$ and $\nu$ are finitely supported probability measures on $G$ and let $\pi$ be the monotone coupling between $\mu$ and $\nu$, with respect to $\ll$. Then, denoting   $\kappa_\pm=\pi\circ{T_\pm}^{-1}$, we have\begin{equation*} H(\mu|m_n)+ H(\nu|m_n)\ge  H(\kappa_-|m_n)+ H(\kappa_+|m_n). \end{equation*} \end{prop}

%Proposition \ref{prop:basic_SCOE} is an immediate consequence of the following proposition. For the proof of this implication, see Section \ref{sec:9to8} below.

\begin{prop}\label{prop:9} Let $(G,\ll)$ be a finitely generated group equipped with a total additive ordering $\ll$.  Let $T_\pm:G\times G\to G$ be complementing operations satisfying  \ref{enu:TEQ } and \eqref{eq:P2'}. Suppose that $\mu$ and $\nu$ are finitely supported probability measures on $G$ and let $\pi$ be the monotone coupling between $\mu$ and $\nu$, with respect to $\ll$. Then, denoting   $\kappa_\pm=\pi\circ{T_\pm}^{-1}$, we have
	$$
	\sum_{(x,y)\in G\times G}\frac{\kappa_-(T_-(x,y))\kappa_+(T_+(x,y))}{\mu(x)\nu(y)}\,\pi(x,y)\le 1.
	$$
\end{prop}

\subsection{The structure of $\supp{\pi}$}

In this section we study the structure of the support of the monotone coupling $\pi$  , denoted by $\supp{\pi}$, under the monotone complementing  operations $T_{-}$ and $T_{+}$, given in Proposition \ref{prop:9}.
%To that end, let $(G,\ll)$ be  a finitely generated group equipped with a total additive ordering $\ll$.  Let $T_\pm:G\times G\to G$ be complementing operations satisfying  \ref{enu:TEQ } and \eqref{eq:P2'}. Suppose that $\mu$ and $\nu$ are finitely supported probability measures on $G$ and let $\pi$ be the monotone coupling between $\mu$ and $\nu$, with respect to $\ll$.

%This will be the key to proving Proposition \ref{prop:9}.
%In particular we define and study the structure of certain chains in $\supp{\pi}$ which will be the key to proving Proposition 2.2.

%\subsubsection{Basic facts} 
%We shall often write $T_\pm$ and $T_\mp$ for brevity, to respectively denote either $T_-$ and $T_+$ or $T_+$ and $T_-$. Similarly, we shall often write $S_\pm$ and $S_\mp$. 
%The following lemmas are symmetric with respect to  the interchangement of $T_-$ with $T_+$. 
%Therefore,  we shall write $T_\pm$ to denote either $T_-$ or $T_+$ and write $T_\mp$ to denote the complementing opertaion for which $T_\pm(x,y)+T_\mp(x,y)=x+y$ for all $x,y\in G$.

\begin{lem}
	\noindent \label{lem: 2.3} Suppose $\left(x_{1},y_{1}\right)\neq\left(x_{2},y_{2}\right)$,
	$x_{1}\preceq x_{2}, y_{1}\preceq y_{2}$ and $T_-\left(x_{1},y_{1}\right)=T_-\left(x_{2},y_{2}\right)$. Then, either $x_{1}=x_{2}$ or $y_{1}=y_{2}$. Moreover, we have  $$T_+(x_1,y_1)\prec T_+\left(x_{2},y_{2}\right)=T_+\left(x_{1},y_{1}\right)+[\left(x_{2}-x_{1}\right)+\left(y_{2}-y_{1}\right)].$$ 
\end{lem}

\begin{proof}
	Assume that $y_1\prec y_2$ and $x_1\prec x_2$. Then, denoting $a=\min\{x_2-x_1,y_2-y_1\}$, we have $$T_-(x_1,y_1)\prec a+T(x_1,y_1)= T_-(x_1+a,y_1+a)\ll T_-(x_2,y_2),$$ a contradiction. Thus, either $x_1=x_2$  or $y_1=y_2$. The relation $T_-(x,y)+T_+(x,y)=x+y$ implies that
	$T_+(x_2,y_2)=x_2+y_2-T_-(x_2,y_2)=x_2+y_2-(x_1+y_1-T_+(x_1,y_1)))$, as claimed.
\end{proof}
Note that Lemma \ref{lem: 2.3} and its proof both hold if one interchanges $T_-$ with $T_+$, due to the symmetry between them.\\

For $a\in G$, define $S_{\pm}\left(a\right)=\left\{ \left(x,y\right)\in\supp{\pi}\,;\,T_{\pm}\left(x,y\right)=a\right\}$.
\begin{lem}
	\noindent \label{lem:S_basic} For every $a\in G$ we have $S_-\left(a\right)=\left\{ \left(x_{0},y_{0}\right),\left(x_{1},y_{1}\right),\dots,\left(x_{k},y_{k}\right)\right\}$ where either $x_{0}=x_{1}=\dots=x_{k}$ and $y_{0}\prec y_{1}\prec\dots\prec y_{k}$ or  $x_{0}\prec x_{1}\prec\dots\prec x_{k}$ and $y_{0}=y_{1}=\dots=y_{k}$. \smallskip
	
	\noindent Similarly, we have $S_+\left(a\right)=\left\{ \left(x'_{0},y'_{0}\right),\left(x'_{1},y'_{1}\right),\dots,\left(x'_{m},y'_{m}\right)\right\}$ where either $x'_{0}=x'_{1}=\dots=x'_{m}$ and $y'_{0}\prec y'_{1}\prec\dots\prec y'_{m}$ or  $x'_{0}\prec x'_{1}\prec\dots\prec x'_{m}$ and $y'_{0}=y'_{1}=\dots=y'_{m}$.
	
%	Moreover, if $\left(x,y\right)\in{\rm supp}\left(\pi\right)$ such that $x_{0}\ll x\ll x_{k}$ and $y_{0}\ll y\ll y_{k}$ then $\left(x,y\right)\in S_{\pm}\left(a\right)$.
\end{lem}

\begin{proof}
By the monotonicity of $\supp{\pi}$, we have 
	$S_-\left(a\right)=\left\{ \left(x_{0},y_{0}\right),\left(x_{1},y_{1}\right),\dots,\left(x_{k},y_{k}\right)\right\}$ where $x_{0}\ll\dots\ll x_{k}$ and $y_{0}\ll\dots\ll y_{k}$. Moreover, by Lemma \ref{lem: 2.3}, for every $i,j\in\left\{ 0\dots,k\right\} $ such that $i\neq j$,
	we have either $x_{i}=x_{j}$ or $y_{i}=y_{j}$, from which it follows that
	either $x_{0}=x_{1}=\dots=x_{k}$ or $y_{0}=y_{1}=\dots=y_{k}$. 
	
The second part of the lemma regarding $S_+(a)$ is proven exactly  in the same way. 

	 %the same as the proof The same proof applies to the second part of the lemma.
	 % is obtained by interchanging  $T_-$ with $T_+$, and repeating first part of the proof verbatim.

\end{proof}

\begin{lem}
	\label{lem:new2.4} Let $a\in G$ and $k\ge2$. Suppose   $S_-\left(a\right)=\left\{ \left(x_{0},y_{0}\right),\dots,\left(x_{k},y_{k}\right)\right\}$ with $x_{0}\ll\dots\ll x_{k}$ and $y_{0}\ll\dots\ll y_{k}$. Then, $S_+\left(T_+\left(x_{i},y_{i}\right)\right)=\left\{ \left(x_{i},y_{i}\right)\right\}$ for all $0<i<k$.
\end{lem}

\begin{proof}
By Lemma \ref{lem: 2.3}, $T_+(x_0,y_0)\prec\dots\prec T_+(x_k,y_k)$. Since $T_+$ is monotone, it  follows that $$T_+(x_0,y_0)\prec\dots\prec T_+(x_k,y_k)\ll T_+(x,y)$$ whenever $x_k\ll x \text{ and } y_k\ll y$. In particular, we have $(x,y)\not\in S_+(T_+(x_i,y_i))$ for all $x_k\ll x$, $y_k\ll y$ and  $0<i<k$. Similarly, one shows that  $(x,y)\not\in S_+(T_+(x_i,y_i))$ for all  $x\ll x_k$, $y\ll y_k$ and $0<i<k$.
%e have
%$$
%x_k\ll x \text{ and } y_k\ll y\implies T_+(x_0,y_0)\prec\dots\prec T_+(x_k,y_k)\ll T_+(x,y),
%$$
%$$
%x_k\ll x \text{ and } y_k\ll y\implies T_+(x,y)\ll T_+(x_0,y_0)\prec\dots\prec T_+(x_k,y_k).
%$$
Since $T_-$ is monotone, if 
$
x_0\ll x\ll x_k \text{ and } y_0\ll y\ll y_k$ then  $T_-(x,y)=a$, and so either $(x,y)=(x_i,y_i)$ for some $0\le i\le k$ or $(x,y)\not\in\supp{\pi}$. Therefore,  by the   monotonicity of $\supp{\pi}$, we have  $S_+(T+(x_i,y_i))=\{(x_i,y_i)\}$ for all $0<i<k$.
\end{proof}

\subsection{Proof of Proposition \ref{prop:9}}
Fix $a\in G$ for which $S_-(a)\neq\emptyset$. It is sufficient to show that 
\begin{equation}\label{eq:one}
\sum_{\left(x,y\right)\in S_{-}(a)}\frac{\kappa_{-}(a)\kappa_{+}\left(T_{+}\left(x,y\right)\right)}{\mu\left(x\right)\nu\left(y\right)}\pi\left(x,y\right)\le\sum_{\left(x,y\right)\in S_{-}(a)}\pi\left(x,y\right)=\kappa_-(a).
\end{equation}
%or equivalently, since $\kappa_-(a)=\sum_{\left(x,y\right)\in S_{-}(a)}\pi\left(x,y\right)$, that 
%\begin{equation*}
%	\sum_{\left(x,y\right)\in S_{-}(a)}\frac{\kappa_{+}\left(T_{+}\left(x,y\right)\right)}{\mu\left(x\right)\nu\left(y\right)}\pi\left(x,y\right)\le1
%\end{equation*}
By Lemma \ref{lem:new2.4}, we have $S_-(a)=\{(x_0,y_0),\dots, (x_k,y_k\}$ where either  $x_0=x_1=\dots=x_k$ and $y_0\prec y_1\prec\dots\prec y_k$ or  $x_0\prec x_1\prec \dots\prec x_k$ and $y_0=y_1=\dots=y_k$. By interchanging the first and
second coordinates, we may assume without loss of generality that the first possiblity holds.

% $x_0=x_1=\dots=x_k$ and $y_0\prec y_1\prec\dots\prec y_k$. 

Therefore,  \eqref{eq:one} is reduced to:
\begin{equation}\label{eq:two}
	\sum_{j=0}^k\frac{\kappa_{+}\left(T_{+}\left(x_0,y_j\right)\right)}{\nu\left(y_j\right)}\pi\left(x_0,y_j\right)\le\mu(x_0).
\end{equation}
 
By Lemma \ref{lem:S_basic}, we have
$
S_+(T_+(x_0,y_0))=\{(x_0,y_0),(x'_0,y'_0),\dots, (x'_l,y'_l)\}
$
such that either  $x'_0=\dots=x'_l=x_0$ or $y'_0=\dots=y'_l=y_0$. In particular, it follows that
\begin{equation}\label{eq:kappa1}
	\kappa_+(T_+(x_0,y_0))\le\max\{\mu(x_0),\nu(y_0))\}.
\end{equation}

Similarly, Lemma \ref{lem:S_basic} tells us that 
$
 S_+(T_+(x_0,y_k))=\{(x_0,y_k),(x''_0,y''_0),\dots, (x''_m,y''_m)\}
$
such that either $x''_0=\dots=x''_m=x_0$ or $y''_0=\dots=y''_m=y_k$. Note that, by lemma \ref{lem: 2.3}, the points 
$$(x_0,y_0),\dots,(x_0,y_k), (x'_0,y'_0),\dots,(x'_l,y'_l),(x''_0,y''_0),\dots,(x''_m,y''_m)
$$
are all distinct.\\

%Similarly, one shows that 
%\begin{equation}\label{eq:kappa2}
%\kappa_+(T_+(x_0,y_k))\le\max\{\mu(x_0),\nu(y_k))\}. 
%\end{equation}

%We split the proof of \eqref{eq:two} into two cases:\\

\noindent \textbf{Case 1: } $k=0$. Using \eqref{eq:kappa1} and the fact that $\pi(x_0,y_0)\le\min\{\mu(x_0),\nu(y_0)\}$ directly yields  \eqref{eq:two}.\\

\noindent \textbf{Case 2: } $k\ge1$. By Lemma \ref{lem:new2.4}, we have 
$S_+(T_+(x_0,y_i)))=\{(x_0,y_i)\}$ for all $0<i<k$, and hence $
\kappa_+(T_+(x_0,y_i))=\pi(x_0,y_i)\le\nu(y_i).$ we can thus rewrite \eqref{eq:two} as follows: 
\begin{equation}\label{eq:three}
\frac{\kappa_+(T_+(x_0,y_0)}{\nu(y_0)}\pi(x_0,y_0)+\sum_{i=1}^{k-1}\frac{\pi(x_0,y_i))}{\nu(y_i)}\pi(x_0,y_i)+\frac{\kappa_+(T_+(x_0,y_k))}{\nu(y_k)}\pi(x_0,y_k)\le \mu(x_0)
\end{equation}

We split the proof of \eqref{eq:three} into four simple subcases, as follows.\\

\noindent \textbf{Case 2.1: } $x'_0=\dots=x'_l=x_0$ and $x''_0=\dots=x''_m=x_0$. Since  $\pi(x,y)\le\nu(y)$ for all $x,y\in G$, it is enough to show that
$$
\kappa_+(T_+(x_0,y_0))+\sum_{i=1}^{k-1}\pi(x_0,y_i)+\kappa_+(T_+(x_0,y_k))\le \mu(x_0)
$$
which cleary holds as $\kappa_+(T_+(x_0,y_0))=\sum_{j=0}^l\pi(x_0,y'_j)$ and $\kappa_+(T_+(x_0,y_k))=\sum_{j=0}^m\pi(x_0,y''_j)$.\\

\noindent \textbf{Case 2.2: } $x'_0=\dots=x'_l=x_0$ and $y''_0=\dots=y''_m=y_0$. Since $\pi(x_0,y_0)\le\nu(y_0)$, 
$\pi(x_0,y_i)\ll\nu(y_i)$ for all $0<i<k$ and 
$$
\kappa_+(T_+(x_0,y_k))=\sum_{j=0}^m\pi(x''_j,y_k)\le\nu(y_k),$$
it is enough to show that 
$$
\kappa_+(T_+(x_0,y_0))+\sum_{i=1}^{k-1}\pi(x_0,y_i)+\pi(x_0,y_k))\le \mu(x_0)
$$
which clearly holds as $\kappa_+(T_+(x_0,y_0))=\sum_{j=0}^l\pi(x_0,y'_j)$.\\

\noindent \textbf{Case 2.3: } $y'_0=\dots=y'_l=y_0$ and $x''_0=\dots=x''_m=x_0$. Since $$\kappa_+(T_+(x_0,y_0))=\sum_{j=0}^l\pi(x'_j,y_0)\le\nu(y_0),
$$
$\pi(x_0,y_i)\ll\nu(y_i)$ for all $0<i<k$ and $\pi(x_0,y_k)\le\nu(y_k)$, it is enough to show that 
$$
\pi(x_0,y_0))+\sum_{i=1}^{k-1}\pi(x_0,y_i)+\kappa_+(T_+(x_0,y_k))\le \mu(x_0)
$$
which clearly holds as $\kappa_+(T_+(x_0,y_k))=\sum_{j=0}^m\pi(x_0,y''_j)$.\\

\noindent \textbf{Case 2.4: } $y'_0=\dots=y'_l=y_0$ and $y''_0=\dots=y''_m=y_0$. Since $$\kappa_+(T_+(x_0,y_0))=\sum_{j=0}^\pi(x'_j,y_0)\le\nu(y_0),
$$
$\pi(x_0,y_i)\ll\nu(y_i)$ for all $0<i<k$ and $\kappa_+(T_+(x_0,y_k))=\sum_{j=0}^m\pi(x''_j,y_k)\le\nu(y_k)$, it is enough to prove that 
$$
\pi(x_0,y_0)+\sum_{i=1}^{k-1}\pi(x_0,y_i)+\pi(x_0,y_k)\le\mu(x_0)
$$
which clearly holds. This establishes \eqref{eq:three} and completes the proof.\qed.

\section{The general case}
\subsection{The Knothe coupling}\label{sec:Knothe}
Denote by $\ZZ^n=(G_i,\ll_i)_{1:k}$ the decomposition of $\ZZ^n$ into a direct sum of groups, $G_1,\dots,G_k$, each of which equipped with a total additive ordering $\ll_i$. 
%For each $i$, and any two finitely supported probability measures $\mu$ and $\nu$ on $G_i$, let $\pi_i$ be the monotone coupling between $\mu$ and $\nu$, defined in the Section \ref{sec:simplified_DCOE}.
We shall next construct the Knothe coupling $\pi$ between two finitely  supported measures $\mu$ and $\nu$ on $\ZZ^n$ with respect to its given decomposition. 

To that end, the following notation shall be useful. For $(x_1,\dots,x_k)\in (G_1,\dots,G_k)$ and each $i\in\{1,\dots,k\}$ denote by $x_{1:i}$ the sub-vector $(x_1,\dots,x_i)\in (G_1,\dots, G_i)$. Consider the disintegration formula for a measure $\kappa$ on $\ZZ^n$ with  respect to the given decomposition:
$$
\kappa(x_1,\dots,x_k)=\kappa^1(x_1)\kappa^2(x_2|x_1)\dots\kappa^k(x_k|x_{1:k-1}),
$$
where $\kappa^1$ is the marginal of $\kappa$ onto $G_1$, $\kappa^2(\cdot\,|\, x_1)$  is the marginal of $\kappa(\cdot \,|\, x_1)$ onto $G_2$ and etc. 

The Knothe coupling between $\mu$ and $\nu$ with respect to this decomposition is defined by
$$\pi(x,y)=\pi^1(x_1,y_1) \pi^2(x_2,y_2 \,|\, x_1,y_1)\dots\pi^k(x_k,y_k \,|\, x_{1:k-1},y_{1:k-1})$$
where $\pi^i(\cdot,\cdot\,|\,x_{1:i-1},y_{1:i-1})$ is the monotone coupling between $\mu^i(\cdot\,|\,x_{1:i-1})$ and $\nu^i(\cdot\,|\,y_{1:i-1})$.
\subsection{Proof of Theorem \ref{thm:Knothe_9}}
%\begin{proof}[Proof of Theorem \ref{thm:Knothe_9}]
Recall that since $T_\pm:\ZZ^n\times \ZZ^n\to \ZZ^n$ are monotone in the sense of Knothe with respect to the decomposition $(G_i,\ll_i)_{1:k}$, we have
$$T_\pm(x,y)=(T_\pm^1(x_1,y_1), T_\pm^2(x_1,x_2,y_1,y_2)\dots ,T_\pm^k(x_{1:k-1},y_{1:k-1}))$$
for all $x=(x_1,\dots,x_k), y=(y_1,\dots,y_k)\in (G_1,\dots,G_k)$. Moreover, for each $i\in\{1,\dots,k\}$,  the operations $(T_\pm^i)^{(x_{1:i-1},y_{1:i-1})}:G_i\times G_i\to G_i$, defined by $$(T_\pm^i)^{(x_{1:i-1},y_{1:i-1})}(x_i,y_i)=T_\pm^i(x_{1:i},y_{1:i})$$
for all $ x_{1:i-1},y_{1:i-1}\in G_1\times\dots\times G_{i-1}$, are increasing in each of their two entries.

Let $\pi$ be the Knothe coupling between $\mu$ and $\nu$ with respect to the same decomposition 
$\ZZ^n=(G_i,\ll_i)_{1:k}$, and recall that  $\kappa_\pm=\pi\circ {T_\pm}^{-1}$.
Then
\begin{align*}
P:&=\sum_{(x,y)\in \ZZ^n\times \ZZ^n}\frac{\kappa_-(T_-(x,y))\kappa_+(T_+(x,y))}{\mu(x)\nu(y)}\,\pi(x,y)\\
&=\sum_{(x_1,y_1)\in G_1\times G_1}\sum_{(x_2,y_2)\in G_2\times G_2}\dots\sum_{(x_k,y_k)\in G_k\times G_k}\frac{\kappa_-(T_-(x,y))\kappa_+(T_+(x,y))}{\mu(x)\nu(y)}\pi(x,y).
\end{align*}
Using the disintegration of  $\kappa_\pm$ and the fact that $T_\pm$ are triangular with respect to the given decomposition of $\ZZ^n$, we have 
$$
\kappa_\pm(T_\pm(x,y))=\kappa^1_\pm(T^1_\pm(x_1,y_1))\kappa_\pm^2(T^2_\pm(x_{2},y_{2})|x_1,y_1)\dots\kappa_\pm^k(T^k_\pm(x_k,y_k)|x_{1:k-1},y_{1:k-1})
$$
where, for brevity, the expression $T^i_\pm(x_i,y_i)$  within $\kappa_\pm^i(T^i_\pm(x_i,y_i)|x_{1:i-1},y_{1:i-1})$  is understood as $(T^i_\pm)^{(x_{1:i-1},y_{1:i-1})}(x_i,y_i)$. Combined with the disintegration of $\mu$, $\nu$, and $\pi$ with respect to the given decomposition of $\ZZ^n$, we obtain that
$$
P=\sum_{(x_1,y_1)\in G_1\times G_1}A_1(x_1,y_1)\sum_{(x_2,y_2)\in G_2\times G_2}A^{(x_1,y_1)}_2(x_2,y_2)\,\,\,\dots\sum_{(x_k,y_k)\in G_k\times G_k}A^{(x_{1:k-1},y_{1:k-1})}_k(x_k,y_k),
$$
where $A^{(x_{1:i-1},y_{1:i-1})}_i(x_i,y_i)$ is given by
\begin{align*}
\frac{\Bp{\kappa_-^i(T^i_-(x_i,y_i)|x_{1:i-1},y_{1:i-1})}\Bp{\kappa_+^i(T^i_+(x_i,y_i)|x_{1:i-1},y_{1:i-1})}}{\Bp{\mu^i(x_i\,|\,x_{1:i-1})}\Bp{\nu^i(y_i|y_{1:i-1})}}\pi^i(x_i,y_i\,|\,x_{1:i-1},y_{1:i-1}).
\end{align*}

Finally, we  apply Proposition \ref{prop:9} iteratively to each sum separately to obtain that 
$$
\sum_{(x_i,y_i)\in G_i}A^{(x_{1:i-1},y_{1:i-1})}_i(x_i,y_i)\le 1
$$
for all $i\in\{1,\dots,k\}$ and $x_{1:i-1},y_{1:i-1}\in G_1\times\dots\times G_{i-1}$. This completes the proof.\qed

\subsection{Proof of Theorem \ref{thm:Knothe}}\label{sec:9to8}
Let $\pi$ be the Knothe coupling between $\mu$ and $\nu$,  given in Theorem \ref{thm:Knothe_9}. By Jensen's inequality, applied to the logarithm function, Theorem \ref{thm:Knothe_9} implies that 
$$
H:=	\sum_{(x,y)\in \ZZ^n\times\ZZ^n}\log\left(\frac{\kappa_-(T_-(x,y))\kappa_+(T_+(x,y))}{\mu(x)\nu(y)}\right)\pi(x,y)\le 0.
$$
By the definition of $\pi$, $\kappa_-$ and $\kappa_+$, it follows that
\begin{align*}
\sum_{z\in\ZZ^n}\log(\kappa_-(z))\kappa_-(z)+\sum_{z\in \ZZ^n}\log(\kappa_+(z))\kappa_+(z)-\sum_{z\in \ZZ^n}\log(\mu(z))\mu(z)-\sum_{z\in \ZZ^n}\log(\nu(z))\nu(z)\le0,
\end{align*}
or, equivalently, that
$H(\kappa_-|m_n)+ H(\kappa_+|m_n)- H(\mu|m_n)- H(\nu|m_n)\le0$. \qed
%which completes the proof. 
\section{Proof of Theorem \ref{thm:strongDBM}}\label{sec:8toBM}
As in \cite{Gozlan2019}, we use the log-Laplace transform of any bounded function $\varphi$:
\begin{equation}\label{eq:log-laplace}
\log \int e^\varphi\,dm_n=\sup_\nu\{\int \varphi\,d\nu-H(\nu|m_n)\}.
\end{equation}
Let $f,g,h,k$ satisfy 
\begin{equation}\label{eq:fghk}
f^\alpha(x)g^\beta(y)\le h^\gamma(T_-(x,y))k^\delta(T_+(x,y))\quad \forall x,y\in \ZZ^n.
\end{equation} 
If either $h$ or $k$ is not bounded from above then the statement holds trivially. Otherwise, it follows from \eqref{eq:fghk} that $f,g,h,k$ are all bounded from above. Given $\eps>0$ and setting
$f_{\eps}=\max(\eps,f(x))$, one may check that the above inequality is equivalent to 
$$
\alpha\log f_{\eps}(x)+\beta\log g_{\eps}(y)\le \gamma\log h_{\eps}(T_-(x,y))+\delta\log k_{\eps}(T_+(x,y)).
$$
Integrating this inequality with respect to the Knothe  coupling $\pi$ between  finitely supported probability measures $\mu,\nu$ on $\ZZ^n$, as given in Theorem \ref{thm:Knothe}, we have
\begin{align*}
\alpha\int\log f_\eps d\mu+\beta\int\log g_\eps d\nu&\le\gamma\int\log (h_\eps\circ T_-)d\pi+\delta\int\log(k_\eps\circ T_+)d\pi\\
&=\int \gamma\log h_\eps d\kappa_-+\delta\int\log k_\eps d\kappa_+,
\end{align*}
where $\kappa_\pm=\pi\circ{T_\pm}^{-1}$. Applying Theorem \ref{thm:Knothe} and \eqref{eq:log-laplace} we thus get
\begin{align*}
\alpha\Bp{\int \log f_\eps d\mu&-H(\mu|m_n)}+\beta\Bp{\int\log g_\eps d\nu -H(\nu|m_n)}\\
&\le \gamma\Bp{\int\log h_\eps d\kappa_--H(\kappa_-|m_n)}+\delta\Bp{\int\log k_\eps d\kappa_+-H(\kappa_+|m_n)}\\ 
&\le \gamma\log \int h_\eps dm_n+\delta\log \int k_\eps dm_n.
\end{align*}
Optimizing over all $\mu$ and $\nu$, we get 
$$
\alpha\log \int f_\eps dm_n+\beta\log \int g_\eps dm_n\le \gamma\log \int h_\eps dm_n+\delta\log \int k_\eps dm_n.
$$
We conclude the proof by taking $\eps\to0$ and applying the monotone convergence theorem.\qed

\bibliographystyle{amsplain_Abr_NoDash}
\bibliography{DBM}

\providecommand{\bysame}{\leavevmode\hbox to3em{\hrulefill}\thinspace}
\providecommand{\MR}{\relax\ifhmode\unskip\space\fi MR }
\providecommand{\MRhref}[2]{%
  \href{http://www.ams.org/mathscinet-getitem?mr=#1}{#2}
}
\providecommand{\href}[2]{#2}
\begin{thebibliography}{10}

\bibitem{Ahlswede1978}
R. Ahlswede and D.~E. Daykin, \emph{{An inequality for the weights of two
  families of sets, their unions and intersections}}, Zeitschrift f{\"{u}}r
  Wahrscheinlichkeitstheorie und Verwandte Gebiete \textbf{43} (1978), no.~3,
  183--185.

\bibitem{AHZ17}
M. Alexander, M. Henk, and A. Zvavitch, \emph{{A discrete version of
  Koldobsky's slicing inequality}}, Israel Journal of Mathematics \textbf{222}
  (2017), no.~1, 261--278.

\bibitem{Cordero-Erausquin2017}
D. Cordero-Erausquin and B. Maurey, \emph{{Some extensions of the
  Pr{\'{e}}kopa-Leindler inequality using Borell's stochastic approach}}, Stud.
  Math. \textbf{238} (2017), no.~3, 201--233.

\bibitem{GG01}
R.~J. Gardner and P. Gronchi, \emph{A brunn-minkowski inequality for the
  integer lattice}, Trans. Amer. Math. Soc. \textbf{353} (2001), 3995--4024.

\bibitem{Gozlan14}
N. Gozlan, C. Roberto, P.-M. Samson, and P. Tetali, \emph{Displacement
  convexity of entropy and related inequalities on graphs}, Probability Theory
  and Related Fields \textbf{160} (2014), no.~1, 47--94.

\bibitem{Gozlan2019}
N. Gozlan, C. Roberto, P.-M. Samson, and P. Tetali, \emph{Transport proofs of
  some discrete variants of the {P}r{\'e}kopa-{L}eindler inequality},
  \href{http://arxiv.org/abs/1905.04038}{arXiv:1905.04038} (2019).

\bibitem{Green2020}
B. Green, D. Matolcsi, I. Ruzsa, G. Shakan, and D. Zhelezov, \emph{{A Weighted
  Pr\'ekopa-Leindler inequality and sumsets with quasicubes}},
  \href{http://arxiv.org/abs/2003.04077}{arXiv:2003.04077} (2020).

\bibitem{HKS2020}
D. Halikias, B. Klartag, and B. Slomka, \emph{{Discrete variants of
  Brunn-Minkowski type inequalities}},
  \href{http://arxiv.org/abs/1911.04392}{arXiv:1911.04392} (2019).

\bibitem{HLY21}
M.~A. {Hern{\'a}ndez Cifre}, E. {Lucas}, and J. {Yepes Nicol{\'a}s}, \emph{{On
  discrete $L_p$ Brunn-Minkowski type inequalities}},
  \href{http://arxiv.org/abs/2105.11441}{arXiv: 2105.11441} (2021).

\bibitem{INZ}
D. Iglesias, J. {Yepes Nicol\'{a}s}, and A. Zvavitch, \emph{Brunn-minkowski
  type inequalities for the lattice point enumerator}, Advances in Mathematics
  \textbf{370} (2020), 107193.

\bibitem{Klartag2019}
B. Klartag and J. Lehec, \emph{{Poisson processes and a log-concave Bernstein
  theorem}}, Stud. Math. \textbf{247} (2019), no.~1, 85--107.

\bibitem{Lindvall02}
T. Lindvall, \emph{Lectures on the coupling method}, Dover Publications, Inc.,
  Mineola, NY, 2002, Corrected reprint of the 1992 original.

\bibitem{MM2020}
A. Marsiglietti and J. Melbourne, \emph{Geometric and functional inequalities
  for log-concave probability sequences},
  \href{http://arxiv.org/abs/2004.12005}{arXiv: 2004.12005} (2020).

\bibitem{Matolcsi2020}
D. Matolcsi, I. Ruzsa, G. Shakan, and D. Zhelezov, \emph{{An analytic approach
  to cardinalities of sumsets}},
  \href{http://arxiv.org/abs/2003.04075}{arXiv:2003.04075} (2020).

\bibitem{Melbourne2020}
J. {Melbourne} and T. {Tkocz}, \emph{Reversal of {R}\'{e}nyi entropy
  inequalities under log-concavity}, IEEE Transactions on Information Theory
  \textbf{67} (2021), no.~1, 45--51.

\bibitem{Ollivier-Villani}
Y. Ollivier and C. Villani, \emph{{A Curved Brunn--Minkowski Inequality on the
  Discrete Hypercube, Or: What Is the Ricci Curvature of the Discrete
  Hypercube?}}, SIAM Journal on Discrete Mathematics \textbf{26} (2012), no.~3,
  983--996.

\bibitem{RYZ17}
D. Ryabogin, V. Yaskin, and N. Zhang, \emph{Unique determination of convex
  lattice sets}, Discrete \& Computational Geometry \textbf{57} (2017), no.~3,
  582--589.

\bibitem{Strassen1965}
V. Strassen, \emph{The existence of probability measures with given marginals},
  Ann. Math. Statist. \textbf{36} (1965), no.~2, 423--439.

\end{thebibliography}
\end{document}